\documentclass[english]{amsart}

\usepackage{enumerate}
\usepackage{amssymb}
\usepackage{array}
\usepackage{graphicx}
\usepackage{mathrsfs}
\usepackage{multicol}
\usepackage{mathtools}
\usepackage{color}
\usepackage{comment}
\usepackage[usenames,dvipsnames,svgnames,table]{xcolor}

\usepackage[curve,all]{xy}
\xyoption{matrix}
 \xyoption{curve}
 \xyoption{color}
 \xyoption{line}
\xyoption{arc}

\usepackage[shortlabels]{enumitem}
\setlist[enumerate]{leftmargin=7mm,topsep=0pt,itemsep=-1ex,partopsep=1ex,parsep=1ex,label=\rm{(\roman*)}}
\setlist[itemize]{leftmargin=5mm,topsep=0pt,itemsep=-1ex,partopsep=1ex,parsep=1ex,label=\raisebox{0.25ex}{\tiny$\bullet$}}

\usepackage[backref=false, colorlinks, linktocpage, citecolor = blue, linkcolor = blue]{hyperref}

\theoremstyle{plain}
\newtheorem{theorem}{Theorem}[section]

\newtheorem*{theoremaux}{Theorem \theoremauxnum}
\gdef\theoremauxnum{1}

\newtheorem*{main-theorem}{Main Theorem}
\newtheorem{proposition}[theorem]{Proposition}
\newtheorem*{propositionaux}{Proposition \propositionauxnum}
\gdef\propositionauxnum{1}

\newtheorem{lemma}[theorem]{Lemma}
\newtheorem*{lemmaaux}{Lemma \lemmaauxnum}
\gdef\lemmaauxnum{1}

\newtheorem{question}[theorem]{Question}
\newtheorem*{questionaux}{Question \questionauxnum}
\gdef\questionauxnum{1}

\newtheorem{corollary}[theorem]{Corollary}

\theoremstyle{definition}
\newtheorem{definition}[theorem]{Definition}

\newtheorem{example}[theorem]{Example}

\theoremstyle{remark}
\newtheorem{remark}[theorem]{Remark}

\makeatletter
\@namedef{subjclassname@2020}{%
  \textup{2020} Mathematics Subject Classification}
\makeatother

\newcommand{\incl}[1][r]{\ar@<-0.2pc>@{^(-}[#1] \ar@<+0.2pc>@{-}[#1]}

\newcommand{\hs}{\kern 0.8pt}

\newcommand{\X}{\mathbb{X}}

\newcommand{\Hom}{\mathrm{Hom}}

\newcommand{\Spec}{\mathrm{Spec}}
\newcommand{\Stab}{\mathrm{Stab}}
\newcommand{\inn}{\mathrm{inn}}

\newcommand{\Lie}{\mathrm{Lie}}

\renewcommand{\SS}{\mathcal{S}}

\renewcommand{\sl}{\mathfrak{sl}}

\newcommand{\so}{\mathfrak{so}}

\renewcommand{\sp}{\mathfrak{sp}}
\newcommand{\e}{\mathfrak{e}}
\newcommand{\f}{\mathfrak{f}}
\newcommand{\g}{\mathfrak{g}}

\newcommand{\lt}{\mathfrak{t}}
\newcommand{\sld}{\mathfrak{sl}_2}
\renewcommand{\l}{\mathfrak{l}}

\newcommand{\z}{\mathfrak{z}}

\newcommand{\RR}{\mathcal{R}}

\newcommand{\gr}{\mathrm{gr}}
\newcommand{\Dyn}{\mathrm{Dyn}}

\renewcommand{\H}{\mathrm{H}}

\newcommand{\A}{\mathbb{A}}
\newcommand{\iso}{\simeq}
\newcommand{\T}{\mathbb{T}}

\newcommand{\C}{\mathbb{C}}

\newcommand{\N}{\mathbb{N}}
\newcommand{\NN}{\mathcal{N}}

\newcommand{\Z}{\mathbb{Z}}
\newcommand{\R}{\mathbb{R}}
\newcommand{\G}{\mathbb{G}}

\renewcommand{\O}{\mathcal{O}}

\newcommand{\codim}{\mathrm{codim}}

\DeclareMathOperator{\Card}{Card}
\DeclareMathOperator{\SL}{SL}

\DeclareMathOperator{\Aut}{Aut}

\DeclareMathOperator{\Gal}{Gal}

\def\ga{\,^\gamma\hskip-1pt}

\title[Real structures on nilpotent orbit closures]{Real structures on nilpotent orbit closures}

\author[Michael Bulois, Lucy Moser-Jauslin, and Ronan Terpereau]{Michael Bulois, Lucy Moser-Jauslin, and Ronan Terpereau}
\thanks{The third-named author is supported by the ANR Project FIBALGA ANR-18-CE40-0003-01.
This work received partial support from the French "Investissements d\textquoteright Avenir" program and from project ISITE-BFC (contract ANR-lS-IDEX-OOOB). The IMB receives support from  the EIPHI Graduate School (contract ANR-17-EURE-0002). Partial support was also received from the program "Investissements d'Avenir" Labex MILYON (contract ANR-10-LABX-0070).
}

\address{Institut Camille Jordan, UMR 5208 CNRS,
Univ Lyon, Universit\'e Jean Monnet, 42023 Saint-Etienne, France}
\email{michael.bulois@univ-st-etienne.fr}

\address{Institut de Math\'{e}matiques de Bourgogne, UMR 5584 CNRS, Universit\'{e} Bourgogne Franche-Comt\'{e}, F-21000 Dijon, France}
\email{lucy.moser-jauslin@u-bourgogne.fr}

\address{Institut de Math\'{e}matiques de Bourgogne, UMR 5584 CNRS, Universit\'{e} Bourgogne Franche-Comt\'{e}, F-21000 Dijon, France}
\email{ronan.terpereau@u-bourgogne.fr}

  	%Linear algebraic groups over the reals, the complexes, the quaternions

\begin{document}

\begin{abstract}
We determine the equivariant real structures on nilpotent orbits and the normalizations of their closures for the adjoint action of a complex semisimple algebraic group on its Lie algebra.
\end{abstract}

\maketitle
\vspace{-5mm}
\tableofcontents

\vspace{-8mm}
\section{Introduction}
Let $G$ be a complex semisimple algebraic group, and let $\g$ be its Lie algebra.
Let $G_0$ be a \emph{real form} of $G$, that is, a real algebraic group such that $G_0 \times_{\Spec(\R)} \Spec(\C) \simeq G$ as complex algebraic groups.
Given a complex quasi-projective $G$-variety $X$, we say that a real $G_0$-variety $X_0$ is an \emph{$(\R,G_0)$-form} of $X$ if $X_0 \times_{\Spec(\R)} \Spec(\C) \simeq X$ as complex $G$-varieties. Let us note that certain $G$-varieties admit no $(\R,G_0)$-form. On the other hand, when an $(\R,G_0)$-form exists, there may be several (and even uncountably many) which are pairwise non-isomorphic. 

\smallskip

The goal of this article is to address the following basic question.

\begin{question}\label{Q1}
What are the $(\R,G_0)$-forms of the nilpotent orbit closures in $\g$ and of their normalizations?
\end{question}

To give a real form $G_0$ of $G$ is equivalent to giving a \emph{real group structure} $\sigma$ on $G$, that is, an antiregular involution on $G$ compatible with the group structure, and giving an $(\R,G_0)$-form of $X$ is equivalent to giving a \emph{$(G,\sigma)$-equivariant real structure} $\mu$ on $X$, that is, an antiregular involution on $X$ such that 
\begin{equation}\label{eq:def eq real structure} \tag{$\star$}
\forall g \in G, \ \forall x \in X,\ \ \mu(g \cdot x)=\sigma(g) \cdot \mu(x). 
\end{equation}
Moreover, two $(\R,G_0)$-forms are isomorphic if and only if the two corresponding $(G,\sigma)$-equivariant real structures are \emph{equivalent}, i.e.~conjugate by a $G$-equivariant automorphism of $X$. 
(See Sections \ref{sec:real group structures} and \ref{sec:equiv real structures} for more details on these notions.)

\smallskip

With this terminology, Question~\ref{Q1} can be rephrased as follows.

\begin{question}\label{Q2}
What are the $(G,\sigma)$-equivariant real structures on the nilpotent orbit closures in $\g$  and on their normalizations?
\end{question}

If $X$ is an \emph{almost homogeneous} $G$-variety, that is, if $X$ contains a dense open $G$-orbit $U$, then the condition~\eqref{eq:def eq real structure} implies that $\mu(U)=U$, and thus $\mu_{|U}$ is a $(G,\sigma)$-equivariant real structure on $U$. 
Therefore, a natural strategy to determine the equivariant real structures on nilpotent orbit closures and their normalizations is in two steps. First, we determine the equivariant real structures on nilpotent orbits. Second, we determine which ones extend to their closures and their normalizations. 

\smallskip

Denoting by $d\sigma_e\colon \g \to \g$ the differential of $\sigma\colon G \to G$ at the identity element, we check (see Example~\ref{ex:dsigma is an eq real structure}) that $d\sigma_e$ is a $(G,\sigma)$-equivariant real structure on $\g$, viewed as a $G$-variety for the adjoint action. In particular, if $\O$ is a nilpotent orbit in $\g$, then $d\sigma_e$ induces a $(G,\sigma)$-equivariant real structure on $\O$ if and only if $d\sigma_e(\O)=\O$.   
However, there are also equivariant real structures on nilpotent orbits that are not obtained by differentiating a real group structure on $G$ (Example \ref{ex:non-gloabl1}), nor even by restricting an equivariant real structure from the Lie algebra $\g$ (Example \ref{ex:non-gloabl2}).

\begin{example}\label{ex:non-gloabl1}
Let $\g$ be a semisimple Lie algebra, and assume that $d\sigma_e(\O)=\O$. Let $\theta \in \R$.
Then $\mu_\theta \colon \O \to \O,\ v \mapsto e^{i\theta} d\sigma_e(v)$ is a $(G,\sigma)$-equivariant real structure on $\O$ which is not obtained by differentiating a real group structure on $G$ when $\theta \notin 2\pi \Z$ (because, in this case, $\mu_\theta$ does not preserve the Lie bracket).
\end{example}

\begin{example}\label{ex:non-gloabl2}
Let $G=\SL_3(\C)$ with $\sigma(g)=\overline{g}$ for all $g \in G$ (here $\overline{g}$ denotes the complex conjugate of $g$), and let $\O_{reg}$ be the \emph{regular} nilpotent orbit in $\sl_3$ (i.e.~the unique nilpotent orbit whose closure contains all the other nilpotent orbits).
Then the map $\mu$ defined by \small 
\[
\O_{reg} \to \O_{reg},\ g \cdot \begin{bmatrix}
0 & 1 & 0 \\ 0 & 0 & 1\\ 0 &0&0
\end{bmatrix} \mapsto \left( \sigma(g)  \begin{bmatrix}
1 & i & 0 \\ 0 & 1 & 0\\ 0 &0&1
\end{bmatrix} \right) \cdot \begin{bmatrix}
0 & 1 & 0 \\ 0 & 0 & 1\\ 0 &0&0
\end{bmatrix} =\sigma(g) \cdot \begin{bmatrix}
0 & 1 & i \\ 0 & 0 & 1\\ 0 &0&0
\end{bmatrix}  
\] \normalsize
is a $(G,\sigma)$-equivariant real structure on $\O_{reg}$ that does not extend to a $(G,\sigma)$-equivariant real structure on $\sl_3$ (see Section \ref{sec:example} for a proof of this claim).
\end{example}

The following theorem provides a complete answer to the second part of Question~\ref{Q2}.
More precisely, we determine the equivariant structures on the nilpotent orbits and the normalizations of their closures. In the cases where the closure is itself not normal, it remains an open question to determine all the equivariant real structures. This will be discussed in Section \ref{sec:extension}.

\begin{main-theorem}%\emph{(See Section \ref{sec:proof of main th} for the proof.)}
Let $G$ be a complex semisimple algebraic group endowed with a real group structure $\sigma$. Let $\g$ be the Lie algebra of $G$, let $\O$ be a nilpotent orbit in $\g$, let $\overline{\O}$ be the closure of $\O$ in $\g$, and let $\widetilde{\O}$ be the normalization of $\overline{\O}$.
Then the following are equivalent:
\begin{enumerate}[$(a)$]
\item\label{item:a} $\O$ is $d\sigma_e$-stable (i.e.~$d\sigma_e(\O)=\O$); 
\item\label{item:b} $\O$ admits a $(G,\sigma)$-equivariant real structure;
\item\label{item:c} $\widetilde{\O}$ admits a $(G,\sigma)$-equivariant real structure; and
\item\label{item:d} $\sigma_D(w(\O))=w(\O)$,
\end{enumerate}
where 
\begin{itemize}
\item $\sigma_D$ is the Dynkin diagram automorphism of $\Dyn(G)$ induced by $\sigma$ (see Definition \ref{def: star-action}); and
\item $w(\O)$ is the weighted Dynkin diagram associated to $\O$ (see \cite[Section 3.5]{CM93}).
\end{itemize}
Moreover, if these equivalent assertions hold, then all $(G,\sigma)$-equivariant real structures on $\O$, resp.~on $\widetilde{\O}$, are equivalent. 
\end{main-theorem}

The Main Theorem will be proved in several steps as follows: 
\begin{itemize}
\item \ref{item:a} $\Leftrightarrow$ \ref{item:b} is Proposition \ref{prop:first part}.
\item \ref{item:b} $\Leftrightarrow$ \ref{item:c} is the first part of Proposition \ref{prop:third part} (see also Section \ref{subsubsec: 2.2.2}).
\item \ref{item:a} $\Leftrightarrow$ \ref{item:d} is Corollary \ref{cor: part_zero}.
\item If the equivalent conditions \ref{item:a}-\ref{item:d} of the Main Theorem hold, then the fact that all $(G,\sigma)$-equivariant real structures on $\O$ are equivalent is Theorem \ref{th:second part}, and the fact that  all $(G,\sigma)$-equivariant real structures on $\widetilde{\O}$ are equivalent is the second part of Proposition \ref{prop:third part}.
\end{itemize} 

\bigskip

Let us make a few comments related to Question~\ref{Q2} and our Main Theorem:
\smallskip
\begin{enumerate}[(1)]
\item In Section \ref{sec:existence}, we provide the list of nilpotent orbits that admit a $(G,\sigma)$-equiva\-riant real structure.
It turns out that, except for a few cases in type $D_{2n}$ (with $n \geq 2$), every nilpotent orbit $\O$ in $\g$  admits a $(G,\sigma)$-equivariant real structure when $\g$ is simple.
\smallskip
\item For every nilpotent orbit $\O$ in a complex simple Lie algebra $\g$, we have $\O^{d\sigma_e}=\O \cap \g^{d\sigma_e}$, which is a real manifold (possibly empty) whose $G_0(\R)$-orbits, usually called \emph{real nilpotent orbits}, are classified (see \cite[Section 9]{CM93}). 
Let us mention that a study of some properties of the complexification process for these real nilpotent orbits can be found in \cite{Cro16}. 
\smallskip
\item When $\overline{\O}$ is non-normal, we do not know whether every $(G,\sigma)$-equivariant real structure on $\O$ extends to $\overline{\O}$; see Example~\ref{ex:example non-normal case}. In particular, we did not manage to provide a complete answer to the first part of Question \ref{Q2}.
A brief review on  what is known about the (non-)normality of nilpotent orbit closures in semisimple Lie algebras can be found in Section \ref{sec:extension}.  
\smallskip
\item There are examples of almost homogeneous varieties for which equivalent real structures on the open orbit extend to inequivalent real structures on the whole variety; see e.g.~the case of $\SL_2(\C)$-threefolds studied in \cite[Section 3]{MJT}.
\end{enumerate}

\subsection*{Notation}
In this article we work over the field of real numbers $\R$ and over the field of complex numbers $\C$.
We denote by
\[\Gamma:=\Gal(\C/\R)=\{1,\gamma\}\]
the Galois group of the field extension $\C/\R$.

We will always denote by $G$ a complex algebraic group, by $\Dyn(G)$ its Dynkin diagram (when $G$ is semisimple), and by $G_0$ a real algebraic group such that $G_0 \times_{\Spec(\R)} \Spec(\C) \simeq G$.
When $X$ is a complex algebraic variety, the group of automorphisms of $X$ over $\Spec(\C)$ is denoted by $\Aut_\C(X)$. Moreover, when $X$ is equipped with a regular $G$-action, the subgroup of $G$-equivariant automorphisms is denoted by $\Aut_\C^G(X)$.
We denote by $\G_{m,\C}$ the multiplicative group over $\C$.

The reader is referred to \cite{CM93} for the general background on Lie theory and nilpotent orbits in semisimple Lie algebras.

\subsection*{Acknowledgment} 
We are grateful to the anonymous referee for the careful reading of the former version of this paper and for his/her helpful comments.

\section{Preliminaries}

\subsection{Real group structures}\label{sec:real group structures}

\subsubsection{} \label{subsubsec:2.1.1}
We start by recalling the notions of real forms and real group structures for a given complex algebraic group $G$, and how these two notions are related.

\begin{definition}\item \label{def:real group structure}
\begin{enumerate}[1.]
\item A \emph{real form} of $G$ is a pair $(G_0,\Theta)$ with $G_0$ a real algebraic group and $\Theta\colon\  G \to G_0 \times_{\Spec(\R)} \Spec(\C)$ an isomorphism of complex algebraic groups. (Most of the time, one drops the isomorphism $\Theta$ and simply write that $G_0$ is a form of $G$.)
\item  A \emph{real group structure} $\sigma$ on $G$ is a scheme involution on $G$ such that the diagram 
\[
\xymatrix@R=4mm@C=2cm{
    G \ar[rr]^{\sigma} \ar[d]  && G \ar[d] \\
    \Spec(\C)  \ar[rr]^{\Spec(z \mapsto \overline{z})} && \Spec(\C)  
  }
  \]
commutes and
 \[\iota_G \circ \sigma=\sigma \circ \iota_G \ \ \text{ and }\ \ m_G \circ (\sigma \times \sigma)=\sigma \circ m_G,\] where $\iota_G\colon G \to G$ is the inverse morphism and $m_G\colon G \times G \to G$ is the multiplication morphism.
\item Two real group structures $\sigma_1$ and $\sigma_2$ on $G$ are \emph{equivalent} if there exists a complex algebraic group automorphism $\psi \in \Aut_{\gr}(G)$ such that \[\sigma_2=\psi \circ \sigma_1 \circ \psi^{-1}.\] 
\end{enumerate}
\end{definition}

There is a correspondence between real group structures on $G$ and real forms of $G$ given as follows (see \cite[Section 2.12]{BS64} for a general statement of Galois descent, and \cite[Section 1.4]{FSS98} for the particular case of algebraic groups).
\begin{itemize}
\item To a real group structure $\sigma$ on $G$, one associates the real algebraic group $G_0:=G/\langle \sigma \rangle$, where $\gamma$ acts on $G$ through $\sigma$, and the isomorphism $\Theta$ is given by $(q,f)$, where $q\colon G \to G_0$ is the quotient morphism and $f\colon G \to \Spec(\C)$ is the structure morphism.
\item To a real form $(G_0,\Theta)$ of $G$, one associates the real group structure $\sigma$ that makes the following diagram commutes:
\[
\xymatrix{
    G \ar[r]^{\Theta \ \ \ \ \ \ \ \ \ \ \ \ }   \ar[d]_{\sigma}  & G_0 \times_{\Spec(\R)}  \Spec(\C)  \ar[d]^{Id \times (\lambda \mapsto \overline{\lambda})} \\
   G \ar[r]^{\Theta \ \ \ \ \ \ \ \ \ \ \ \ }  &  G_0 \times_{\Spec(\R)}  \Spec(\C)
  }
\]
\end{itemize}
Moreover, two real forms of $G$ are isomorphic (as real algebraic groups) if and only if the corresponding real group structures are equivalent.

\subsubsection{}
In this article, we are mostly interested in complex semisimple algebraic groups; nevertheless, the case of complex tori will play a role in the proof of Theorem~\ref{th:second part}.

\begin{lemma} \label{lem: real form on tori} 
(Real group structures on complex tori; see \cite[Lemma~1.5]{MJT18} and \cite[Theorem~2]{Cas08}.)\ 
Let $T \iso \G_{m,\C}^{n}$ be an $n$-dimensional complex torus with $n \geq 1$.
\begin{enumerate}[leftmargin=*]
\item If $n=1$, then $T$ has exactly two inequivalent real group structures, defined by $\sigma_0: t \mapsto \overline{t}$ and $\sigma_1: t \mapsto \overline{t}^{-1}$.
\item If $n=2$, then $\sigma_2:  (t_1,t_2) \mapsto (\overline{t_2},\overline{t_1})$ defines a real group structure on $T$.
\item \label{item: n at least 2} If $n\ge 2$,  then every real group structure on $T$ is equivalent to exactly one real group structure of the form  $\sigma_0^{\times n_0}\times\sigma_1^{\times n_1}\times\sigma_2^{\times n_2}$, where $n=n_0+n_1+2n_2$.
\end{enumerate}
\end{lemma}

\subsubsection{}
Let $G$ be a complex semisimple algebraic group. There exists a central isogeny $\widetilde{G} \to G$, where $\widetilde{G}$ is a simply-connected semisimple algebraic group. Then $\widetilde{G}$ is isomorphic to a product of simply-connected simple algebraic groups (see \cite[Exercise~1.6.13 and Section 6.4]{Con14}).
Moreover, every real group structure $\sigma$ on $G$ lifts uniquely to a real group structure $\widetilde{\sigma}$ on $\widetilde{G}$ (this follows for instance from \cite[Section 3.4, Theorem]{Pro07}).

The next result reduces the classification of real group structures on simply-connected semisimple algebraic groups to the classification of real group structures on simply-connected simple algebraic groups.

\begin{lemma} \emph{(\cite[Lemma~1.7]{MJT18})}\label{lem:easy_lemma_reduction}
Let $\sigma$ be a real group structure on a complex simply-connected semisimple algebraic group $\widetilde{G} \iso \prod_{i \in I} G_i$, where the $G_i$ are the simple factors of $\widetilde{G}$. Then, for a given $i \in I$, we have the following possibilities:
\begin{enumerate}[leftmargin=*]
\item \label{item: Gi stable}$\sigma(G_i)=G_i$ and $\sigma_{|G_i}$ is a real group structure on $G_i$; or
\item \label{item:Gi and Gj swap}  there exists $j \neq i$ such that $\sigma(G_i)=G_j$, then $G_i \iso G_j$ and $\sigma_{| G_i \times G_j}$ is equivalent to $(g_1,g_2) \mapsto (\sigma_0(g_2),\sigma_0(g_1))$  for an arbitrary real group structure $\sigma_0$ on $G_i \iso G_j$.
\end{enumerate}
\end{lemma}

Real group structures on complex simply-connected simple algebraic groups are well-known (see e.g.~\cite[Section 1.7.2]{GW09} or \cite[Section V\!I.10]{Kna02}).
Therefore, all real group structures on complex (simply-connected) semisimple algebraic groups can be determined from Lemma \ref{lem:easy_lemma_reduction}.

\subsubsection{} Let $G$ be a complex semisimple algebraic group endowed with a real group structure $\sigma$.
In order to describe which nilpotent orbits in $\g$ admit a $(G,\sigma)$-equivariant real structure (in Section \ref{sec:existence}), we need to recall the notions of split real group structures, as well as the notions of inner and outer twists.
\begin{definition}\item
\begin{enumerate}[1.]
\item If there exists a Borel subgroup $B \subseteq G$ such that $\sigma(B)=B$, then $\sigma$ is called \emph{quasi-split}.
If furthermore $B$ contains a maximal torus $T$ such that $\sigma(T)=T$ and the restriction $\sigma_{|T}$ is equal to the product $\sigma_0^{\times \dim(T)}$ (with the notation of Lemma~\ref{lem: real form on tori}), then $\sigma$ is called \emph{split}.
\item For $c \in G$, we write 
\[\inn_c\colon G \to G,\ g \mapsto cgc^{-1}.\] We say that $\inn_c$ is an \emph{inner} automorphism of $G$. An automorphism of $G$ that is not inner is called \emph{outer}.
If $\sigma_1$ and $\sigma_2$ are two real group structures on $G$, we may write $\sigma_2=\varphi \circ \sigma_1$ with $\varphi \in \Aut_{\gr}(G)$, and we say that $\sigma_2$ is an \emph{inner} (resp.~an \emph{outer}) \emph{twist} of $\sigma_1$ if $\varphi$ is an inner (resp.~an outer) automorphism of $G$. 
Note that the relation of being inner twists is an equivalence relation on the set of real group structures on $G$.
\end{enumerate}
\end{definition}

According to \cite[Proposition~7.2.12]{Con14}, every real group structure on $G$ is an inner twist of a unique (up to equivalence) quasi-split real group structure on $G$. Moreover, in types $A_1-B-C-E_7-E_8-F_4-G_2$ (i.e.~when $\Aut(\Dyn(G))$ is trivial), any quasi-split real group structure is equivalent to a split one, and in types $A_n (n \geq 2)-D-E_6$ (i.e.~when $\Aut(\Dyn(G))$ contains exactly one conjugacy class of elements of order $2$), there are exactly two equivalence classes of quasi-split real group structures; this follows for instance from \cite[Propositions~7.2.2 and 7.2.12]{Con14}.

\subsubsection{}\label{subsubsec: star-action}
Finally, we recall how the choice of a real group structure $\sigma$ on a complex semisimple algebraic group $G$ induces a $\Gamma$-action on $\Dyn(G)$. We refer to \cite[Remark 7.1.2]{Con14} or \cite[Section A.2]{MJT18} and references therein for details. 

Let $B \subseteq G$ be a Borel subgroup, and let $T \subseteq B$ be a maximal torus.
(We say that $(T,B)$ is a \emph{Borel pair} in $G$.)
Then $(T',B'):=(\sigma(T),\sigma(B))$ is another Borel pair in $G$, and so there exists $c_{\sigma} \in G$ (unique up to left multiplication by an element of $T$) such that 
\[
c_{\sigma} T' c_{\sigma}^{-1}=T\ \ \ \ \text{and}\ \ \ \ c_{\sigma} B' c_{\sigma}^{-1}=B.
\]
Let $\theta:=\inn_{c_{\sigma}} \circ \sigma$ that is an antiregular automorphism of $G$. 
Since $\theta(T)=T$ and $\theta(B)=B$, the automorphism $\theta$ induces a lattice automorphism of the character group $\X(T):=\Hom_{\gr}(T,\G_{m,\C})$ as follows:
 \[
 \forall \chi \in \X(T),\ \chi \mapsto\ga \chi :=\sigma_0 \circ \chi \circ \theta^{-1},\]
where $\sigma_0(t)=\overline{t}$ is the complex conjugation on $\G_{m,\C}$. One can check that this lattice automorphism is of order $\leq 2$, that it does not depend on the choice of $c_{\sigma}$, and that it stabilizes the sets of roots $\RR \subseteq \X(T)$ and of simple roots $\SS \subseteq \RR$ associated with the triple $(G,B,T)$.

\begin{definition}\label{def: star-action}
By the previous discussion, the choice of $\sigma$ induces an action $\Gamma \to \Aut(\Dyn(G))$, usually referred as the \emph{$\star$-action induced by $\sigma$}. This action does not depend on the choice of the Borel pair $(T,B)$ in $G$.
We will denote by $\sigma_D \in \Aut(\Dyn(G))$ the image of the non-trivial element $\gamma \in \Gamma$.
\end{definition} 

\begin{remark}\label{rk: star actions and inner twists}
It follows from the definition of the $\star$-action that if $\sigma$ is an inner twist of $\sigma'$, then the $\star$-actions induced by $\sigma$ and $\sigma'$ coincide. 
In particular, if $\sigma$ is an inner twist of a split real group structure on $G$, then $\sigma_D$ is trivial. On the other hand, if $\sigma$ is an outer twist of a split real group structure on $G$, then it follows from \cite[Section 1.5, second half of p.~41]{Con14} that $\sigma_D$ is non-trivial.
\end{remark}

\subsection{Equivariant real structures}  \label{sec:equiv real structures}
We fix a complex algebraic group $G$, a real group structure $\sigma$ on $G$, and we denote by $G_0=G/\langle \sigma \rangle$ the corresponding real form.

\subsubsection{} We start by recalling the notions of real forms and equivariant real structures for complex $G$-varieties, and how they are related, as we did for complex algebraic groups in Section \ref{subsubsec:2.1.1}.

\begin{definition}\label{def: equivariant real structure}
Let $X$ be a complex $G$-variety.
\begin{enumerate}[1.]
\item An \emph{$(\R,G_0)$-form} of $X$ is a pair $(X_0,\Xi)$ with $X_0$ a real $G_0$-variety and $\Xi\colon X \to X_0 \times_{\Spec(\R)} \Spec(\C) $ an isomorphism of complex $G$-varieties. 
(As for real forms of complex algebraic groups, one usually drops the isomorphism $\Xi$ and simply write that $X_0$ is a form of $X$.) 
\item A \emph{$(G,\sigma)$-equivariant real structure} on $X$ is an antiregular involution $\mu$ on $X$, that is, a scheme involution on $X$  such that the following diagram commutes
\[
\xymatrix@R=4mm@C=2cm{
    X \ar[rr]^{\mu} \ar[d]  && X \ar[d] \\
    \Spec(\C)  \ar[rr]^{\Spec(z \mapsto \overline{z})} && \Spec(\C)  
  }
  \]
and satisfying the condition~\eqref{eq:def eq real structure} stated in the introduction.
\item Two $(G,\sigma)$-equivariant real structures $\mu$ and $\mu'$ on $X$ are \emph{equivalent} if there exists a $G$-equivariant automorphism $\varphi \in \Aut_\C^G(X)$ such that $\mu'=\varphi \circ \mu\circ \varphi^{-1}$.
\end{enumerate}
\end{definition}

We now suppose that $X$ is quasi-projective. (Later in this article, we will in fact restrict to the case where $X$ is quasi-affine.) Under this condition, as for real group structures (see Section \ref{subsubsec:2.1.1}), there is a correspondence between $(G,\sigma)$-equivariant real structures on $X$ and $(\R,G_0)$-forms of $X$ (see \cite[Section 5]{Bor20}). Moreover, this correspondence induces a bijection between isomorphism classes of $(\R,G_0)$-forms of $X$ (as real algebraic $G_0$-varieties) and equivalence classes of $(G,\sigma)$-equivariant real structures on $X$.

\begin{example}\label{ex:dsigma is an eq real structure}
Considering the Lie algebra $\g$ of $G$ as a $G$-variety for the adjoint action, we observe that the antilinear Lie algebra involution $d\sigma_e$, obtained by differentiating $\sigma$ at the identity element, is also a $(G,\sigma)$-equivariant real structure on $\g$.
Indeed, for a fixed $g \in G$, if we denote $\inn_g\colon G \to G, h \mapsto ghg^{-1}$, then we have 
\[\inn_{\sigma(g)} \circ \sigma=\sigma \circ \inn_g,\] 
and the condition \eqref{eq:def eq real structure} follows from computing the differential at the identity element on both sides of this equality. 
\end{example}

\subsubsection{}\label{subsubsec: 2.2.2}
The condition \eqref{eq:def eq real structure} implies that an equivariant real structure $\mu$ on a $G$-variety $X$ maps $G$-orbits to $G$-orbits, and so a given $G$-orbit is either fixed by $\mu$ or exchanged with another $G$-orbit. In particular, if $X$ has a dense open orbit $U$, then $U$ must be fixed by $\mu$, and so $\mu_{|U}$ is an equivariant real structure on $U$. Let us note that a given equivariant real structure $\mu_0$ on $U$ need not extend to $X$, but if $\mu_0$ extends to $X$, then this extension is unique by \cite[Chapter I, Lemma~4.1]{Har}.

\subsubsection{}
In the case of homogeneous spaces, we have the following criterion for the existence of an equivariant real structure.

\begin{lemma} \emph{(\cite[Lemma~2.4]{MJT18})}\label{lem: two conditions}
Let  $X=G/H$ be a homogeneous space.
Then $X$ admits a $(G,\sigma)$-equivariant real structure if and only if there exists $g_0 \in G$ such that the following two conditions hold:
\begin{enumerate}[1.]
\item \label{eq: sigma compatible}
\emph{$(G,\sigma)$-compatibility condition:}  $\sigma(H)=g_0Hg_{0}^{-1}$
\item \label{eq: involution}
\emph{involution condition:}\hspace{15mm} $\sigma(g_0)g_0 \in H$
\end{enumerate}
in which case one real structure is given by 
\[ \forall k \in G, \ \mu(kH)=\sigma(k)g_0H.\]
\end{lemma}

\subsubsection{}
The classical way to determine the number of equivalence classes for real structures, assuming that such a structure exists, is via Galois cohomology (see \cite{Ser02} for a general reference on Galois cohomology).

\begin{lemma}  \emph{(\cite[Lemma~2.11]{MJT18}, see also \cite[Section 8]{Wed18})}\label{lem:Galois H1 to param eq real structures}\\
Let  $X$ be a $G$-variety endowed with a $(G,\sigma)$-equivariant real structure $\mu_0$. 
Let the Galois group $\Gamma$ act on $\Aut_\C^G(X)$ by conjugation by $\mu_0$, and denote by $\H^1(\Gamma,\Aut_\C^G(X))$ the corresponding first Galois cohomology set.
Then the map
{\small
\[\begin{array}{ccl}
 \H^1(\Gamma,\Aut_\C^G(X)) &\to &\{ \text{equivalence classes of $(G,\sigma)$-equivariant real structures on $X$}\}\\
 \varphi &\mapsto &\ \ \ \varphi \circ \mu_0
 \end{array}\]}
 \normalsize
 is a bijection that sends the identity element to the equivalence class of $\mu_0$.
\end{lemma}

\subsection{Nilpotent orbits in semisimple Lie algebras}
In this section we recall some classical facts concerning nilpotent orbits in semisimple Lie algebras that we will need to prove the Main Theorem. 

\smallskip

Let $G$ be a complex semisimple algebraic group, and let $\g$ be its Lie algebra.
We denote by $\NN$ the \emph{nilpotent cone} of $\g$, that is, the $G$-stable closed subset of $\g$ formed by the elements $v \in \g$ such that $\mathrm{ad}_v\colon \g \to \g, u \mapsto [u,v]$, is a nilpotent endomorphism of the $\C$-vector space $\g$. Then the following hold:
\smallskip
\begin{enumerate}[(N1)]
\item\label{item:N1} The nilpotent cone $\NN$ is a finite union of $G$-orbits (see \cite[Corollary~3.2.15]{CM93}).\smallskip
\item\label{item:N2} The dimension of a nilpotent orbit is always even (see \cite[Corollary~1.4.8]{CM93}).\smallskip
\item\label{item:N3} Each nilpotent orbit is stable for the $\G_{m,\C}$-action induced by the scalar multiplication on $\g$ (see the proof in \cite[Lemma~4.3.1]{CM93}). \smallskip
\item\label{item:N4} To every nilpotent orbit $\O \subseteq \g$ is associated a unique weighted Dynkin diagram $w(\O)$, that is, a labeling of the nodes of the Dynkin diagram $\Dyn(G)$ with labels in $\{0,1,2\}$ (see \cite[Section 3.5]{CM93}). When $\g$ is a complex simple Lie algebra, a complete classification of the weighted Dynkin diagrams corresponding to nilpotent orbits can be found in \cite[Section 5.3]{CM93}, for the classical cases, and in \cite[Section 8.4]{CM93}, for the exceptional cases.
\smallskip
 \item\label{item:N5} If $v \in \NN$, then $(G \cdot v) \cap \z(\g^{v}) \subseteq \g$ is a dense open subset of $\z(\g^{v})$, where $\g^v:=\{ u \in \g \ |\ [u,v]=0\}$ and $\z(\g^{v})$ is the center of $\g^v$ (see \cite[Propositions~35.3.3 and~35.3.4]{TY05}). 
\smallskip
\item \label{item:N6} Let $v_1,v_2 \in \NN$. Then $\Stab_{G}(v_1)$ and $\Stab_{G}(v_2)$ are conjugate in $G$ if and only if $v_1$ and $v_2$ belong to the same nilpotent orbit (see \cite[Theorem~3.7.1]{Br97}, using the fact that each nilpotent orbit is a decomposition class). 
\smallskip
\item \label{item:N7} Since the Lie algebra $\g$ is semisimple, it is isomorphic to a direct sum of simple Lie algebras $\bigoplus_{i \in I} \g_i$ on which $G$ acts componentwise through the adjoint group $G/Z(G)$. 
Then the nilpotent orbits in $\g$ identify with the products $\prod_{i \in I} \O_i$, where each $\O_i$ is a nilpotent orbit in $\g_i$. The same holds for the nilpotent orbit closures in $\g$ and their normalizations. 
\end{enumerate}

\section{Existence of real structures on nilpotent orbits}\label{sec:existence}
Let $G$ be a complex semisimple algebraic group endowed with a real group structure $\sigma$. 
Let $(T,B)$ be a Borel pair in $G$.
Let $\g$ be the Lie algebra of $G$, and let $\O$ be a nilpotent orbit in $\g$.

\smallskip

In this section we prove the equivalences \ref{item:a} $\Leftrightarrow$ \ref{item:b}(Proposition \ref{prop:first part}) and  \ref{item:a} $\Leftrightarrow$ \ref{item:d} (Corollary \ref{cor: part_zero}) of the Main Theorem. In particular, the equivalence \ref{item:b} $\Leftrightarrow$ \ref{item:d} gives a combinatorial criterion  for the existence of an equivariant real structure on a given nilpotent orbit, and we then use this criterion to determine which nilpotent orbits admit an equivariant real structure (see Proposition \ref{prop:fixed_orbits} and the discussion before).

\begin{proposition}\label{prop:first part}\item
\begin{enumerate}[$(1)$]
\item\label{item: existence_i} If $d\sigma_e(\O)=\O$, then $(d\sigma_e)_{|\O}$ is a $(G,\sigma)$-equivariant real structure on $\O$.
\item\label{item: existence_ii} If $d\sigma_e(\O) \neq \O$, then $\O$ does not admit a $(G,\sigma)$-equivariant real structure.
\end{enumerate}
\end{proposition}

\begin{proof}
\ref{item: existence_i}: Follows from Example~\ref{ex:dsigma is an eq real structure}.\\
\ref{item: existence_ii}: 
Let $v \in \O$ and $v':=d\sigma_e(v) \in \O':=d\sigma_e(\O)$. Let $H:=\Stab_G(v)$ and $H':=\Stab_G(v')$.
Since $d\sigma_e$ is a $(G,\sigma)$-equivariant real structure on $\g$, the condition~\eqref{eq:def eq real structure} holds for every element in $\g$; in particular, the condition~\eqref{eq:def eq real structure} applied to $v$ implies that $\sigma(H)=H'$.
Let us assume that $\O$ admits a $(G,\sigma)$-equivariant real structure.
By Lemma~\ref{lem: two conditions}\ref{eq: sigma compatible}, there exists $g_0 \in G$ such that $\sigma(H)=g_0Hg_{0}^{-1}$, and so $H$ and $H'=\sigma(H)$ are conjugate. 
But, according to \ref{item:N6}, the subgroups $H$ and $H'$ are conjugate if and only if $\O=\O'$. The result follows by contraposition.
\end{proof}

\begin{proposition}\label{prop: part_zero}
Let $w(\O)$ be the weighted Dynkin diagram associated to $\O$. 
Then
\[w(d\sigma_e(\O))=\sigma_D(w(\O)),\]
where $\sigma_D$ is the Dynkin diagram automorphism induced by $\sigma$ (see Definition \ref{def: star-action}).
\end{proposition}

\begin{proof}
We keep the notation of Section \ref{subsubsec: star-action}.
Let $\lt \subseteq \g$ be the Cartan subalgebra associated to $T \subseteq G$.  
Given a nilpotent orbit $\O\subseteq \g$, we recall the construction of the corresponding weighted Dynkin diagram $w(\O)$ (see \cite[Section 3.5]{CM93} for details). There exists an $\sld$-triple $(x,t,y)$ with $x\in \O$, $t\in \lt$, and such that for every simple root $\alpha\in \SS$, we have $d\alpha_e(t)\in \N$, where $d\alpha_e \in \lt^*$ is the differential of the simple root $\alpha$ at the identity element. In fact, for a given $\O$, there is a unique $t \in \lt$ satisfying this last condition, and one can furthermore show that $d\alpha_e(t)\in \{0,1,2\}$.
Then $w(\O)$ is defined by associating to each node $\alpha \in \SS$ of $\Dyn(G)$ the label $d\alpha_e(t)\in \{0,1,2\}$. 

Let us now consider the weighted Dynkin diagram $\sigma_D(w(\O))$. By definition of $\sigma_D$, it is the weighted Dynkin diagram defined by associating to each node $\alpha \in \SS$ of $\Dyn(G)$ the label 
\begin{equation}\label{eq: gamma alpha} \tag{$\diamond$}
d(\ga \alpha)_e(t)=d\alpha_e(d\theta^{-1}_e(t))\in \{0,1,2\},
\end{equation} 
where $d\theta^{-1}_e$ is the antilinear Lie algebra automorphism obtained by differentiating $\theta^{-1}=\sigma \circ \inn_{c_\sigma^{-1}}\colon G \to G$ at the identity element.
 On the other hand, if $(x,t,y)$ is an $\sld$-triple as above, then 
 \[(x',t',y'):=(d\theta^{-1}_e(x),d\theta^{-1}_e(t),d\theta^{-1}_e(y))\] is an  $\sld$-triple with $x' \in d\theta^{-1}_e(\O)=d\sigma_e(c_\sigma^{-1} \cdot \O)=d\sigma_e(\O)$, $t' \in \lt$ (since $\theta(T)=T$), and such that for every simple root $\alpha\in \SS$, we have $d\alpha_e(t')\in \N$ by \eqref{eq: gamma alpha}. It follows that $w(d\sigma_e(\O))=\sigma_D(w(\O))$.
\end{proof}

The following result is a straightforward consequence of Proposition \ref{prop: part_zero} together with \ref{item:N4}.

\begin{corollary}\label{cor: part_zero}
We have $d\sigma_e(\O)=\O$ if and only if $\sigma_D(w(\O))=w(\O)$.
\end{corollary}

We now determine which nilpotent orbits in $\g$ admit an equivariant real structure (i.e. the nilpotent orbits $\O$ such that $d\sigma_e(\O)=\O$).
Applying Lemma~\ref{lem:easy_lemma_reduction}, we reduce to the case where either 
\begin{enumerate}[label=(\Alph*)]
\item\label{case i} $\g$ is simple; or 
\item\label{case ii} $\g\simeq \l \oplus \l$ with $\l=\Lie(L)$ a simple summand, and $\sigma\colon L \times L \to L \times L$ is  
$(g_1,g_2) \mapsto (\sigma'(g_2),\sigma'(g_1))$, where $\sigma'$ is any real group structure on $L$.
\end{enumerate}

\smallskip

In Case~\ref{case ii}, the nilpotent orbits of $\g$ stabilized by $d\sigma_e$ are the orbits of the form $(\O_1,d\sigma'_e(\O_1))$ with $\O_1$ a nilpotent orbit in $\l$. Hence, only these nilpotent orbits admit a $(G,\sigma)$-equivariant real structure. 

\smallskip

On the other hand, Case~\ref{case i} will be handled by Proposition \ref{prop:fixed_orbits}, but we first need to recall the partition type classification of nilpotent orbits in type $D_n$.

\begin{theorem}\emph{(\cite[Theorem 5.1.4]{CM93})}
Nilpotent orbits in $\so_{2n}$ (type $D_n$ with $n \geq 4$) are parametrized by partitions of $2n$ in which even parts occur with even multiplicity, except that ``very even" partitions $\mathbf{d}=[d_1,\ldots,d_r]$ (those with only even parts, each having even multiplicity) correspond to two orbits, denoted by $\O_{\mathbf{d}}^{I}$ and $\O_{\mathbf{d}}^{I\!I}$.
\end{theorem}

\begin{proposition}\label{prop:fixed_orbits}
Assume that $G$ is simple. Then $d\sigma_e(\O) \neq \O$ if and only if one of the following two conditions holds:
\begin{enumerate}[1.]
\item $G$ is of type $D_4$, $\sigma$ is an outer twist of a split real group structure on $G$ (i.e.~$\g^{d\sigma_e}\simeq \so_{5,3}$ or $\so_{7,1}$), and $\O$ belongs to one of the pairs of orbits swapped by the $\star$-action induced by $\sigma$ (see Example \ref{ex:3.5} below).
\item $G$ is of type $D_{2n}$ with $n \geq 3$, $\sigma$ is an outer twist of a split real group structure on $G$ (i.e.~$\g^{d\sigma_e}\simeq \so_{p,q}$ with $p$ and $q$ odd), and $\O \in \{ \O_{\mathbf{d}}^{I}, \O_{\mathbf{d}}^{I\!I} \}$ with $\mathbf{d}$ a very even partition.
\end{enumerate}
\end{proposition}

\begin{proof}
\begin{itemize}
\item If $\sigma$ is an inner twist of a split real group structure on $G$, then the corresponding Dynkin diagram automorphism $\sigma_D$ is trivial (see Remark \ref{rk: star actions and inner twists}). 
Hence, Corollary \ref{cor: part_zero} yields $d\sigma_e(\O)=\O$ in this case.
\item In types $A_n\, (n \geq 2)$, $D_{2n+1}\, (n \geq 2)$ and $E_6$, the weighted Dynkin diagram $w(\O)$ is invariant under the non-trivial Dynkin diagram automorphism; see \cite[Lemma 3.6.5]{CM93} for the type $A_n$, \cite[Lemma 5.3.4 and Remark 5.3.6]{CM93} for the type $D_{2n+1}$, and \cite[Section 8.4]{CM93} for the case $E_6$. Hence, Corollary \ref{cor: part_zero} yields again $d\sigma_e(\O)=\O$ in this case.
\item In type $D_{2n}$ ($n \geq 2$), if $\sigma$ is an outer twist of a split real group structure on $G$, then $\sigma_D$ is a non-trivial involution of $\Dyn(G)$ (see Remark \ref{rk: star actions and inner twists}). In type $D_4$, the group $\Aut(\Dyn(G)) \simeq \mathfrak{S}_3$ contains three (conjugate) involutions, and $\sigma_D$ can be each of them depending on the choice of $\sigma$ in its equivalence class. In type $D_{2n}$ ($n \geq 3$), we have $\Aut(\Dyn(G))=\langle \sigma_D\rangle  \simeq \Z/2\Z$. The result follows then from the description of the weighted Dynkin diagrams given in \cite[Example 5.3.7]{CM93} in type $D_4$ and in \cite[Lemmas 5.3.4 and 5.3.5, Remark 5.3.6]{CM93} in type $D_{2n}$ ($n \geq 3$).
\end{itemize}
\end{proof}

\begin{example}\label{ex:3.5}
A partial order relation is defined on the set of nilpotent orbits of $\g$ by inclusion of closures.
In type $D_4$, the Hasse diagram representing this partial order is the following (see \cite[Section 6.2]{CM93}).
\begin{multicols}{2}
\begin{center}
\scalebox{0.8}{
\xymatrix@R=4mm@C=2cm{
& \O_{[7,1]} \ar@{-}[d] & \\
& \O_{[5,3]} \ar@{-}[rd] \ar@{-}[ld] \ar@{-}[d] &\\
\O_{[4^2]}^{\mathrm{I}} \ar@{-}[rd] \ar@/^3pc/@{<.>}[rr] \ar@{<.>}[r] & \O_{[5,1^3]} \ar@{<.>}[r]\ar@{-}[d]& \O_{[4^2]}^{\mathrm{I\!I}} \ar@{-}[ld]\\
& \O_{[3^2,1^2]} \ar@{-}[d]&  \\
& \O_{[3,2^2,1]} \ar@{-}[ld] \ar@{-}[rd] \ar@{-}[d]&  \\
\O_{[2^4]}^{\mathrm{I}} \ar@/^3pc/@{<.>}[rr] \ar@{<.>}[r]\ar@{-}[rd] & \O_{[3,1^5]} \ar@{<.>}[r]\ar@{-}[d]& \O_{[2^4]}^{\mathrm{I\!I}}  \ar@{-}[ld] \\  
& \O_{[2^2,1^4]}\ar@{-}[d]&  \\
& \O_{[1^8]}&   
      }
            }
\end{center}
Here the dotted arrows indicate which pairs of orbits can be swapped by the $\star$-action induced by an outer twist of a split real group structure on $G$.
\end{multicols}
\end{example}

\section{Uniqueness of real structures on nilpotent orbits}\label{sec:uniqueness}
We keep the notation of Section \ref{sec:existence}.
In this section, we prove that all $(G,\sigma)$-equivariant real structures on $\O$ are equivalent (Theorem \ref{th:second part}).

\smallskip

We recall that, if $X \simeq G/H$ is a homogeneous space, then there is an isomorphism of groups (see \cite[Proposition~1.8]{Tim11}) given by
\[
N_G(H)/H \xrightarrow{\sim} \Aut_\C^{G}(X),\ nH \mapsto (\varphi_{nH}\colon gH \mapsto gn^{-1}H),
\] 
where $N_G(H)$ denotes the normalizer of $H$ in $G$.

Let $\O \simeq \prod_{i \in I} \O_i \simeq G/H$ be a nilpotent orbit in $\g \simeq \bigoplus_{i \in I} \g_i$, and let
$Q:=N_G(H)/H \simeq \Aut_\C^G(\O)$.
The following result, which is essentially due to Brylinski and Kostant \cite{BK91}, describes the structure of the algebraic group $Q$.

\begin{theorem}\label{th:structure of Q}
The linear algebraic group $Q$ is connected and solvable.
Furthermore, the unipotent radical $U$ of $Q$ is of codimension 
\[
d=d(\O):=\Card(\{ i \in I,\ \O_i \neq \{0\}   \})
\] 
in $Q$. In particular, we have $\T:=Q/U \simeq \G_{m,\C}^{d}$ and $Q \simeq U \rtimes \T$.
\end{theorem}

\begin{proof}
Let $v \in \O$ such that $\Stab_G(v)=H$. Then $Q \simeq Q \cdot v$ as algebraic varieties.
Also, by \cite[Theorem~D]{BK91}, we have $Q \cdot v =\O \cap \z(\g^v)$, which is a dense open subset of $\z(\g^v)$ by \ref{item:N5}. Hence, $Q\simeq Q \cdot v$ is irreducible, and so $Q$ is connected.

Now the fact that $Q$ is solvable and that its unipotent radical is of codimension $d$ in $Q$ is precisely the content of \cite[Theorem~23 and Remark~23.1]{BK91}. The last sentence follows from general results on the structure of connected solvable linear algebraic groups (see e.g.~\cite[Section 10.6]{Bor91}).
\end{proof}

\begin{theorem}\label{th:second part}
All $(G,\sigma)$-equivariant real structures on $\O$ are equivalent.
\end{theorem}

\begin{proof}
If $\O$ admits no $(G,\sigma)$-equivariant real structure, then there is nothing to prove.
Otherwise, according to Proposition~\ref{prop:first part}, $\mu_0:=(d\sigma_e)_{|\O}$ is a $(G,\sigma)$-equivariant real structure on $\O$. Thus, conjugation by $\mu_0$ yields a $\Gamma$-action on the group $\Aut_\C^G(\O)\simeq Q$, and according to Lemma~\ref{lem:Galois H1 to param eq real structures}, the first Galois cohomology set $\H^1(\Gamma,Q)$ parametrizes the equivalence classes of $(G,\sigma)$-equivariant real structures on $\O$. Therefore proving the theorem is equivalent to prove that this set is a singleton.

Let $K:=\G_{m,\C}^{d}$ acting on $\g \simeq \bigoplus_{i \in I} \g_i$ by scalar multiplication on each simple component $\g_i$ such that $\O_i \neq \{0\}$. 
By \ref{item:N3}, the torus $K$ acts faithfully on $\O$, and so there is an inclusion of algebraic groups $K \hookrightarrow \Aut_\C^G(\O) \simeq Q$. Using the fact that $\mu_0=d\sigma_e$ is an antilinear Lie algebra involution that permutes the $\g_i$, we have the following two possibilities (with the notation of Lemma \ref{lem: real form on tori}): 
\begin{itemize}
\item Assume that $\O_i \neq \{0\}$ and $\mu_0(\g_i)=\g_i$. In this case, $\Gamma$ acts on the corresponding $\G_{m,\C}$ through $\sigma_0$. Indeed, denoting $\varphi_t\colon \g_i \to \g_i, \ v \mapsto t \cdot v$, we see that  \[
\mu_{0|\g_i} \circ \varphi_t \circ \mu_{0| \g_i}=\varphi_{\overline{t}}.
\]
\item Assume that $\O_i, \O_j \neq \{0\}$ and $\mu_0(\g_i \oplus \g_j)=\g_j \oplus \g_i$. Then $\g_i \oplus \g_j \simeq \l \oplus \l$, for some simple Lie algebra $\l$, and $\mu_{0|\g_i \oplus \g_j}$ identifies with $ \l \oplus \l \to \l \oplus \l,\ (v_1,v_2) \mapsto (\psi(v_2),\psi^{-1}(v_1))$, where $\psi$ is an antilinear Lie algebra automorphism. In this case, $\Gamma$ acts on the corresponding $\G_{m,\C}^{2}$ through the real group structure $\sigma_2$. 
Indeed, denoting $\varphi_{(t_1,t_2)} \colon \l \oplus \l \to \l \oplus \l, \ (v_1,v_2) \mapsto (t_1 \cdot v_1, t_2 \cdot v_2)$, we see that  \[
\mu_{0|\g_i \oplus \g_j} \circ \varphi_{(t_1,t_2)} \circ \mu_{0 | \g_i \oplus \g_j}=\varphi_{(\overline{t_2},\overline{t_1})}.
\]
\end{itemize}
Therefore, the $\Gamma$-action on $\Aut_\C^G(\O)$, given by the conjugation by $\mu_0$, stabilizes the subgroup $K$, and the restriction of the $\Gamma$-action on $\Aut_\C^G(\O)$ to $K=\G_{m,\C}^{d}$ coincides with the $\Gamma$-action on $K$ induced by a real group structure equivalent to $\sigma_{0}^{\times n_0} \times \sigma_{2}^{\times n_2}$ for some $n_0,n_2 \in \Z_{\geq 0}$ such that $d=n_0+2n_2$. But then \cite[Proposition~1.18]{MJT18} implies that $\H^1(\Gamma,K)$ is trivial.

Moreover, the homomorphism $Q \to Q,\ q \mapsto \gamma \cdot q$ is an antiregular algebraic  group involution, and so the $\Gamma$-action on $Q$ stabilizes the unipotent radical $U$ of $Q$ and induces a $\Gamma$-action on the quotient $\T=Q/U$ such that the quotient morphism $\pi\colon Q \to \T$ is then $\Gamma$-equivariant. It follows that isomorphism of algebraic groups $\pi_{|K}\colon K \simeq \T$ is $\Gamma$-equivariant. Consequently, it induces an isomorphism of abelian groups between $\H^1(\Gamma,K)$ and $\H^1(\Gamma,\T)$, and so  $\H^1(\Gamma,\T)$ is trivial.

Finally, the exact sequence of algebraic groups given by Theorem~\ref{th:structure of Q}, namely
\[
1 \to U \to Q \to \T \to 1,
\] 
induces a long exact sequence of pointed sets (see \cite[Chp.~I, Section 5.5]{Ser02}). In particular, we have
\[
\H^1(\Gamma,U) \to \H^1(\Gamma,Q) \to \H^1(\Gamma,\T) = \{ e \}.
\]
But $\H^1(\Gamma,U)$ is a singleton by \cite[Chp.~I\!I\!I, Section 2.1, Proposition 6]{Ser02}, and so $\H^1(\Gamma,Q)$ is also a singleton, which concludes the proof of the theorem.
\end{proof}

\section{Extension of real structures to nilpotent orbit closures}\label{sec:extension}
We maintain the notation of Section \ref{sec:existence}.
In this section, we prove that every $(G,\sigma)$-equivariant real structure on $\O$ extends uniquely to $\widetilde{\O}$, and that all $(G,\sigma)$-equivariant real structures on $\widetilde{\O}$ are equivalent (Proposition \ref{prop:third part}).
We then briefly review what is known about the (non-)normality of nilpotent orbit closures in semisimple Lie algebras, and finish this section by considering an example related to our extension problem (Example \ref{ex:example non-normal case}). 

\begin{lemma}\label{lem:real structures always  extend}
Let $X$ be a complex normal affine $G$-variety, and assume that there exists a $G$-stable dense open subset $X_0 \subseteq X$ such that $\codim(X \setminus X_0,X) \geq 2$. Then any $G$-equivariant automorphism of $X_0$ extends uniquely to a $G$-equivariant automorphism of $X$.
\end{lemma}

\begin{proof}
It is clear that if an automorphism extends from $X_0$ to $X$, then this extension is unique by \cite[Chapter I, Lemma~4.1]{Har}.

Since $X$ is normal and $\codim(X \setminus X_0,X) \geq 2$, the map 
\[\H^0(X,\O_X) \to \H^0(X_0,\O_{X_0}),\ f \mapsto f_{|X_0}
\] 
is an isomorphism of $\C$-algebras, and so $\Spec(\H^0(X_0,\O_{X_0})) \simeq \Spec(\H^0(X,\O_X)) \simeq X$.
By \cite[\href{https://stacks.math.columbia.edu/tag/01P9}{Tag 01P9}]{stacks-project}, the affinization morphism 
\[X_0 \to \Spec(\H^0(X_0,\O_{X_0})) \simeq X\] 
is an open immersion; it coincides with the inclusion morphism $X_0 \hookrightarrow X$.

Let $\varphi$ be an automorphism of $X_0$. Then it follows that 
\[
\tilde{\varphi}:=\Spec(\varphi^*) \colon X \simeq \Spec(\H^0(X_0,\O_{X_0})) \to \Spec(\H^0(X_0,\O_{X_0})) \simeq X
\]
is an automorphism of $X$ that extends (uniquely) $\varphi$. Furthermore, the set
\[
\bigcap_{g \in G} \{ x \in X\ | \ \tilde{\varphi}(g \cdot x)=g \cdot \tilde{\varphi}(x)\}
\]
is a closed subset of $X$ containing the dense open subset $X_0$, and so it is equal to $X$. Therefore, $\tilde{\varphi}$ is a $G$-equivariant automorphism of $X$. 
\end{proof}

\begin{proposition}\label{prop:third part}
Every $(G,\sigma)$-equivariant real structure on $\O$ extends uniquely to the normalization $\widetilde{\O}$ of $\overline{\O}$. 
Furthermore, all $(G,\sigma)$-equivariant real structures on $\widetilde{\O}$ are equivalent.
\end{proposition}

\begin{proof}
If $\mu$ is a $(G,\sigma)$-equivariant real structure on $\O$, then by Theorem~\ref{th:second part}, there exists $\varphi \in \Aut_\C^G(\O)$ such that $\mu=\varphi \circ (d\sigma_e)_{|\O} \circ \varphi^{-1}$. The equivariant real structure $(d\sigma_e)_{|\O}$ is the restriction of $d\sigma_e\colon \g \to \g$, so it extends (uniquely by \cite[Chapter~I, Lemma~4.1]{Har}) to $\overline{\O}$, and also to $\widetilde{\O}$ by the universal property of the normalization morphism (see e.g.~\cite[Section 12.11]{GW09}). With a slight abuse of notation, we will write $(d\sigma_e)_{|\widetilde{\O}}$ to denote the extension of $(d\sigma_e)_{|\O}$ to $\widetilde{\O}$. 

In addition, $\varphi \in \Aut_\C^G(\O)$ also extends (uniquely) to $\tilde{\varphi} \in \Aut_\C^G(\widetilde{\O})$ by Lemma~\ref{lem:real structures always  extend} applied to $X=\widetilde{\O}$ and $X_0=\O$ (here the fact that $\codim(X \setminus X_0,X) \geq 2$ follows from \ref{item:N2}).
Hence, $\tilde{\mu}:=\tilde{\varphi} \circ (d\sigma_e)_{|\widetilde{\O}} \circ \tilde{\varphi}^{-1}$ is a $(G,\sigma)$-equivariant real structure on $\widetilde{\O}$ extending $\mu$. The fact that this extension is uniquely defined by $\mu$ follows again from \cite[Chapter I, Lemma~4.1]{Har}.

It remains to prove the second part of the proposition. 
Let $\widetilde{\mu}$ be any $(G,\sigma)$-equivariant real structure on $\widetilde{\O}$. Then its restriction to $\O$, denoted by $\mu$, is equivalent to $(d\sigma_e)_{|\O}$ by Theorem~\ref{th:second part}.
Hence, there exists $\varphi \in \Aut_\C^G(\O)$ such that $\mu=\varphi \circ (d\sigma_e)_{|\O} \circ \varphi^{-1}$. 
By Lemma~\ref{lem:real structures always  extend}, $\varphi$ extends to $\widetilde{\varphi} \in \Aut_\C^G(\widetilde{\O})$, and so we must have $\widetilde{\mu}=\widetilde{\varphi} \circ (d\sigma_e)_{|\widetilde{\O}} \circ \widetilde{\varphi}^{-1}$. In particular, 
$\widetilde{\mu}$ is equivalent to $(d\sigma_e)_{|\widetilde{\O}}$. Therefore, all $(G,\sigma)$-equivariant real structures on $\widetilde{\O}$ are equivalent.
\end{proof}

\smallskip

The question of which nilpotent orbits in a complex simple Lie algebra have normal closure has been studied by many authors. Let us note that, since a product variety is normal if and only if each factor is normal, the (non-)normality of $\overline{\O} \simeq \prod_{i \in I} \overline{\O}_i$ (see \ref{item:N7}) is therefore easily deduced from the (non-)normality of the $\overline{\O}_i$, which means that we can reduce to the case where $\g$ is simple.   
Here is a list of some known results on this topic, with references:
\begin{itemize}
\item The regular orbit, the subregular orbit, and the minimal orbit have normal closure \cite{Kos63,VP72,Bro94}.
\item Nilpotent orbits in $\sl_n$ have normal closure \cite{Hes79,KP79}. 
\item Nilpotent orbits in $\so_n$ and $\sp_{2n}$ having normal closure are determined in \cite{KP82,Som05}.
\item Nilpotent orbits in $\g_2$, $\f_4$, and $\e_6$ having normal closure are classified in \cite{LS88,Kra89,Bro98b,Som03}. A list of nilpotent orbits in $\e_7$ and $\e_8$ whose closure is non-normal, which is expected to be the complete list, is given in \cite[Section~7.8]{Bro98a}.
\end{itemize}

\bigskip

We now give an example of a quasi-affine surface $X_0$ in $\A_{\C}^{4}$, endowed with a real structure $\mu$ that does not extend to the closure $X$ of $X_0$ in $\A_{\C}^{4}$ but does extend to the normalization $\widetilde{X}$ of $X$.

\begin{example}\label{ex:example non-normal case}
Let $X$ be the image of the morphism
\[
f\colon \A_{\C}^{2} \to \A_{\C}^{4},\ (s,t) \mapsto (s,st,t^2,t^3);
\]
it is a closed non-normal affine surface in $\A_{\C}^{4}$, defined by the prime ideal 
\[(u^2w-v^2,vw-ux,uw^2-vx,w^3-x^2) \ \text{in}\  \C[u,v,w,x].
\]
One can check that $X$ has an isolated singularity at $p:=(0,0,0,0)$, and that $f$ restricts to an isomorphism $\A_{\C}^{2} \setminus \{(0,0)\} \to X_0:=X \setminus p$. Moreover, the normalization of $X$, which coincides with the affinization of $X_0$, is $\widetilde{X}=\A_{\C}^{2}$, and the normalization morphism corresponds to the inclusion of $\C$-algebras $\C[y,yz,z^2,z^3] \subseteq \C[y,z]$.

Let $\tilde{\mu}$ be the real structure on $\widetilde{X}$ given by $(s,t) \mapsto (\overline{t},\overline{s})$. It restricts to a real structure $\mu$ on $X_0$ (as $(0,0)$ is fixed by $\widetilde{\mu})$. But the corresponding comorphism  
\[
\mu^*\colon \C[y,z] \to \C[y,z],\ Q(y,z) \mapsto \overline{Q}(z,y)
\]
does not preserve $\C[y,yz,z^2,z^3]$ (as $\mu^*$ exchanges $y$ and $z$), and so $\mu$ does not extend to a real structure on $X$.
\end{example}

We do not know whether a similar phenomenon can occur when $X_0$ is a nilpotent orbit with non-normal closure. 
More precisely, by the results we have proven, we know that any $(G,\sigma)$-equivariant real structure on a nilpotent orbit is equivalent to one that extends (whether the closure of the orbit is normal or not), namely $d\sigma_e$. However, in the case where the nilpotent orbit does not have a normal closure, it is \textit{a priori} possible that we have two equivalent real structures, one that extends, and one that does not, or even two that extend to inequivalent real structures.

\section{Back to Example \ref{ex:non-gloabl2}}\label{sec:example}
In this last section, we justify that the map $\mu$, given in Example \ref{ex:non-gloabl2}, is a $(G,\sigma)$-equivariant real structure on $\O_{reg}$ that does not extend to a $(G,\sigma)$-equivariant real structure on $\sl_3$. (Of course, it follows from our Main Theorem that $\mu$ is equivalent to a $(G,\sigma)$-equivariant real structure that does extend to $\g$, namely $(d\sigma_e)_{|\O_{reg}}$, but we will see that this does not imply that $\mu$ itself extends to $\g$.)

\smallskip

Recall that, in this example, $G=\SL_3(\C)$, and $\sigma(g)=\overline{g}$ for all $g \in G$. To check that $\mu$ is indeed a $(G,\sigma)$-equivariant real structure on $\O_{reg}$, it suffices to check that the element $g_0:=\begin{bmatrix}
1 & i & 0\\ 0 & 1 & 0\\0 & 0 & 1
\end{bmatrix}$ satisfies the two conditions of Lemma \ref{lem: two conditions} with 
\[
H:=\Stab_G\left(\begin{bmatrix}
0 & 1 & 0\\ 0 & 0 & 1\\0 & 0 & 0
\end{bmatrix}\right)=\left\{ \begin{bmatrix}
a & b & c\\ 0 & a & b\\0 & 0 & a
\end{bmatrix};\ a,b,c \in \C \ \text{with}\ a^3=1\right\},
\]
which is a straightforward computation left to the reader.
 
Let $\mu_0:=(d\sigma_e)_{|\O_{reg}}$, and let $\varphi:=\mu_0 \circ \mu \in \Aut_\C^G(\O_{reg})$. 
We check that 
\[
\varphi\colon 
\O_{reg} \to \O_{reg},\ g \cdot \begin{bmatrix}
0 & 1 & 0 \\ 0 & 0 & 1\\ 0 &0&0
\end{bmatrix} \mapsto \left( g \begin{bmatrix}
1 & -i & 0 \\ 0 & 1 & 0\\ 0 &0&1
\end{bmatrix}  \right) \cdot \begin{bmatrix}
0 & 1 & 0 \\ 0 & 0 & 1\\ 0 &0&0
\end{bmatrix} = 
g  \cdot \begin{bmatrix}
0 & 1 & -i \\ 0 & 0 & 1\\ 0 &0&0
\end{bmatrix}. 
\]
Since $\mu_0$ is the restriction of $d\sigma_e\colon \sl_3 \to \sl_3$, we see that $\mu$ extends to a $(G,\sigma)$-equivariant real structure on $\sl_3$ if and only if there exists $\Phi \in \Aut_\C^G(\sl_3)$ such that $\Phi_{|\O_{reg}}=\varphi$.

Suppose that such a $\Phi$ exists. Then,  for any $u\in\sl_3$, the isotropy subgroups of $u$ and $\Phi(u)$ are the same. 
Let $u_{a,b}:= \begin{bmatrix}
a & b & 0 \\ 0 & a & 0\\ 0 &0&-2a
\end{bmatrix}$ with $a,b\in\C^*$.
Then 
 $
\left\{ \begin{bmatrix}
\alpha & \beta & 0 \\ 0 & \alpha & 0\\ 0 &0&\alpha^{-2}\\
\end{bmatrix}\ \middle|\ \alpha\in\C^*,\beta\in\C \right\}$
is the isotropy subgroup of $u_{a,b}$ in $G$.
Moreover, we check that $\{u_{a,b}\ | \ a,b\in\C^*\}$ are the only elements in $\sl_3$ with this isotropy subgroup, and so this set must be stable by $\Phi$. 
 
Let $f_1,f_2$ be two $\C$-scheme endomorphisms of $\A_\C^1 \setminus \{0\}$ such that $\Phi(u_{a,1})=u_{f_1(a),f_2(a)}$. It is well-known that $f_i(a)=\nu_i a^{\epsilon_i}$ with $\nu_i \in \C^*$ and $\epsilon_i \in \Z$. Let $w_a:=\begin{bmatrix} a & 1 & 0 \\ 0 & a & 1 \\ 0 & 0 & -2a
\end{bmatrix}$ and let $g_a:=\frac{1}{3}\begin{bmatrix}
a^{-1} & 0 & 3 \\ 0 &a^{-1}& -9a \\ 0 & 0 & 27a^2 
\end{bmatrix}$. Then we check that $w_a=g_a \cdot u_{a,1}$, and thus
\begin{align*}
\begin{bmatrix}
0 & 1 & -i \\ 0 & 0 & 1\\ 0 &0&0
\end{bmatrix} &=\Phi\left( \begin{bmatrix}
0 & 1 & 0 \\ 0 & 0 & 1\\ 0 &0&0
\end{bmatrix}  \right)=\Phi\left( \lim_{a \to 0} w_a \right)=\lim_{a \to 0} \Phi\left(  w_a \right)
=\lim_{a \to 0} \left( g_a \cdot \Phi\left(  u_{a,1} \right) \right)\\
&=\lim_{a \to 0} \left( g_a \cdot u_{f_1(a),f_2(a)} \right)=\lim_{a \to 0} \begin{bmatrix}
 \nu_1 a^{\epsilon_1}& \nu_2 a^{\epsilon_2}& \frac{1}{3}(\nu_2 a^{\epsilon_2-1} -\nu_1 a^{\epsilon_1-2})\\
 0  &  \nu_1 a^{\epsilon_1} & \nu_1 a^{\epsilon_1-1} \\
 0 &0 & -2\nu_1 a^{\epsilon_1}
\end{bmatrix}.
\end{align*}
This implies readily that $\epsilon_1=1, \epsilon_2=0, \nu_1=\nu_2=1$, but then $\nu_2 a^{\epsilon_2-1} -\nu_1 a^{\epsilon_1-2}=0$, which is absurd. Hence, there is no $\Phi \in \Aut_\C^G(\sl_3)$ such that $\Phi_{|\O_{reg}}=\varphi$, and so the $(G,\sigma)$-equivariant real structure $\mu$ given in Example \ref{ex:non-gloabl2} is not obtained by restricting a $(G,\sigma)$-equivariant real structure from $\sl_3$ to $\O_{reg}$.

\bibliographystyle{alpha}
\bibliography{biblio}

\end{document}